\documentclass[11pt]{amsart}
\usepackage{calc,amssymb,amsmath,ulem,amsfonts,mathrsfs}
\usepackage{alltt}
\RequirePackage[dvipsnames,usenames]{color}

\input{kmacros3.sty}
\input{xy}
\xyoption{all}
\usepackage{hyperref}
\usepackage{graphicx}
\usepackage[all,cmtip]{xy}
\usepackage{verbatim}
\usepackage[left=1.25in,top=1.25in,right=1.25in,bottom=1.25in]{geometry}
\usepackage{mabliautoref}
\usepackage{mathrsfs}

\renewcommand{\frm}{\mathfrak{m}}
\renewcommand{\m}{\mathfrak{m}}
\newcommand{\idealb}{\mathfrak{b}}

\newcommand{\icolon}[2]{ \left({#2}:{#1} \right)}
\newcommand{\NN}{\mathbb{N}}

\newcommand{\gs}{\geqslant}
\newcommand{\RR}{\mathbb{R}}
\newcommand{\ds}{\displaystyle}
\newcommand{\coeff}{\mathrm{coeff}}

\makeatletter
\let\@wraptoccontribs\wraptoccontribs
\makeatother

\maketheorem{setup}{Setup}{theorem}

\begin{document}
\numberwithin{equation}{theorem}

\title[Behavior at the $F$-pure threshold]{On the behavior of singularities at the $F$-pure threshold}


\thanks{Daniel Hern\'andez was supported in part by the NSF Postdoctoral Fellowship \#1304250}
\thanks{Karl Schwede was supported in part by NSF FRG Grant DMS \#1265261/1501115, NSF CAREER Grant DMS \#1252860/1501102 and a Sloan Fellowship.}
\thanks{Alessandro De Stefani was supported in part by NSF Grant DMS \#1259142.}
\thanks{Robert Walker was supported in part by NSF GRF under \#PGF-031543 and by NSF RTG grant
\#0943832.}
\thanks{Emily Witt was supported in part by NSF Grant DMS \#1501404.}

\begin{abstract}
We provide a family of examples where the $F$-pure threshold and the log canonical threshold of a polynomial are different, but where $p$ does not divide the denominator of the $F$-pure threshold (compare with an example of \mustata-Takagi-Watanabe).  We then study the $F$-signature function in the case where either the $F$-pure threshold and log canonical threshold coincide or where $p$ does not divide the denominator of the $F$-pure threshold.  We show that the $F$-signature function behaves similarly in those two cases.  Finally, we include an appendix which shows that the test ideal can still behave in surprising ways even when the $F$-pure threshold and log canonical threshold coincide.
 \end{abstract}

\author[E.~Canton]{Eric Canton}
\address{Department of Mathematics,
University of Nebraska--Lincoln,
203 Avery Hall,
Lincoln, NE  68588}
\email{ecanton2@math.unl.edu}

\author[D.~Hern\'andez]{Daniel J. Hern\'andez}
\address{Department of Mathematics,
University of Utah,
155 S 1400 E,
Salt Lake City, Utah 84112}
\email{dhernan@math.utah.edu}

\author[K.~Schwede]{Karl~Schwede}
\address{Department of Mathematics,
University of Utah,
155 S 1400 E,
Salt Lake City, Utah 84112}
\email{schwede@math.utah.edu}

\author[E.~Witt]{Emily E.~Witt}
\address{Department of Mathematics,
University of Utah,
155 S 1400 E,
Salt Lake City, Utah 84112}
\email{witt@math.utah.edu}

\contrib[With an appendix by]{Alessandro De Stefani}
\address{P. O. Box 400137,
Dept. of Mathematics,
University of Virginia,
319 Kerchof Hall,
Charlottesville, VA 22904}
\email{ad9fa@virginia.edu}

\contrib[]{Jack Jeffries}
\address{
2844 East Hall,
Department of Mathematics,
University of Michigan,
530 Church Street,
Ann Arbor, MI 48109}
\email{jeffries@math.utah.edu}

\contrib[]{Zhibek Kadyrsizova}
\address{
2844 East Hall,
Department of Mathematics,
University of Michigan,
530 Church Street,
Ann Arbor, MI 48109}
\email{zhikadyr@umich.edu}

\contrib[]{Robert Walker}
\address{
2844 East Hall,
Department of Mathematics,
University of Michigan,
530 Church Street,
Ann Arbor, MI 48109}
\email{robmarsw@umich.edu}

\contrib[]{George Whelan}
\address{
Department of Mathematical Sciences,
4400 University Drive, MS3 F2,
George Mason University,
Fairfax, Virginia, 22030
}
\email{gwhelan@masonlive.gmu.edu}


\maketitle

\section{Introduction}


Inspired by connections between singularities of the minimal model program and those from tight closure theory, S.\;Takagi and K.\,i.\;Watanabe introduced the \emph{$F$-pure threshold} \cite{TakagiWatanabeFPureThresh}; see also \cite{MustataTakagiWatanabeFThresholdsAndBernsteinSato}.
If $f$ is a nonzero element in a Noetherian ring $R$ of prime characteristic, then the $F$-pure threshold, denoted $\fpt(f)$, is the largest positive real number $t$ for which the pair $(R, f^t)$ is $F$-pure. 
This number has been shown to be rational in several contexts; for regular rings, the focus of this paper, rationality is proven in \cite{BlickleMustataSmithFThresholdsOfHypersurfaces}.

The $F$-pure threshold is closely related to the \emph{log canonical threshold}, an important measure of the singularities that
has appeared in several guises \cite[Sections 8-10]{KollarSingularitiesOfPairs}.  The log canonical threshold of an element $f$ is denoted $\lct(f)$, and though this invariant is often only considered when the characteristic of the ambient space is zero, it is, in fact, defined in all characteristics.
Moreover, if $R$ is a polynomial ring over $\bQ$, and $f_p$ is the reduction of $f \in R$ modulo $p$ \cite{HochsterHunekeTightClosureInEqualCharactersticZero},
then \cite{HaraWatanabeFRegFPure} and \cite[Corollary 4.1]{ZhuLCTPosChar}
imply that
\begin{equation} \label{inequalities: e}
\fpt(f_p) \leq \lct(f_p) \leq \lct(f).
\end{equation}
Furthermore, standard reduction to characteristic $p > 0$ techniques enable one to show that $\lct(f_p) = \lct(f)$ for $p \gg 0$ \cite{HaraYoshidaGeneralizationOfTightClosure}.


The values in \eqref{inequalities: e} coincide for $f = y^2-x^3 \in \bC[x,y]$ when $p \equiv 1 \bmod 6$ in which case $\fpt(f_p) = \lct(f_p) = \lct(f) = {5 \over 6}$, but $\fpt(f_p) = {5 \over 6} - {1 \over 6p}$ if $p \equiv 5 \bmod 6$.  This type of behavior seems common for many singularities.  In fact,
it is conjectured that there always exists a Zariski-dense set of primes $p$ for which \eqref{inequalities: e} consists of equalities \cite{MustataSrinivasOrdinary}, and this motivates understanding the $F$-pure threshold when these numbers do \emph{not} coincide.

In numerous examples, including the cusp $f = y^2-x^3$ above, it has been noted that when $\fpt(f_p) \neq \lct(f_p)$,
it is frequently the case that $p$ divides the denominator of $\fpt(f_p)$.
It was even asked if this was always the case (including by the third author of this paper).
There is one example in the literature of a polynomial $f$ for which certain reductions $f_p$ satisfy $\fpt(f_p) \neq \lct(f_p)$, but $p$ does not divide the denominator of $\fpt(f_p)$; see \cite[Example 4.5]{MustataTakagiWatanabeFThresholdsAndBernsteinSato}.  This example, however, is not as widely known as it should be.

On the other hand, in many cases, $p$ must divide the denominator of $\fpt(f_p)$ whenever
$\fpt(f_p) \neq \lct(f_p)$.  More precisely, this occurs for diagonal polynomials \cite{HernandezFInvariantsOfDiagonalHyp}, binomials \cite{HernandezFPureThresholdOfBinomial}, homogeneous polynomials with an isolated singularity \cite{BhattSinghCalabiYau, HernandezNunezWittZhangHomogeneous}, and all homogeneous polynomials in two variables \cite{HernandezTeixeiraTwoVariable}.  We begin by providing a new family of examples for which $\fpt(f_p) \neq \lct(f_p)$, but $p$ does not divide the denominator of $\fpt(f_p)$.

\begin{theoremA*}[\autoref{prop.PNotDividingDenominatorExample}, \autoref{cor.LCTComputed}]
Fix a prime $p > 2$ and an $F$-finite field $k$ containing $\bF_p$.  Suppose that $d \geq n \geq 3$, $d > 3$, and  $p \nmid d (n(d-2) - d)$.  If
\[f = x_1^d + \cdots + x_n^d + (x_1 \cdots x_n)^{d-2} \in R = k[x_1, \ldots, x_n],\] then $\lct(f) = {n  \over d}$.
If, in addition, $p \equiv -1 \bmod d$, then $\fpt(f) = {n (p-d+1)+d \over d(p-1)}$.

In particular, there exist infinitely many primes $p$ for which $\fpt(f) \neq \lct(f)$, yet $p$ does not divide the denominator of $\fpt(f)$.
\end{theoremA*}

\begin{remark}
A forthcoming paper of the second and fourth authors will explore large classes of polynomials for which $\fpt(f) \neq \lct(f)$, but where $p$ does not divide the denominator of the $F$-pure threshold.
\end{remark}

If $p$ does not divide the denominator of $\fpt(f)$, the $F$-singularities of the pair $(R, f^{\fpt(f)})$ are similar to the $F$-singularities of $F$-pure rings.
For example, in both cases, the test ideal is reduced and cuts out an $F$-pure scheme \cite{VassilevTestIdeals,SchwedeSharpTestElements}.
On the other hand, if $p$ divides the denominator of the $F$-pure threshold, then the test ideal of the pair need not even be reduced \cite{MustataYoshidaTestIdealVsMultiplierIdeals}.  Thus, when searching for conditions that guarantee that the pair $(R, f^{\fpt(f)})$ is ``well behaved,'' there are at least two clear candidates:
\begin{enumerate}
\item  The characteristic does not divide the denominator of the $F$-pure threshold.
\label{cond1}
\item The $F$-pure threshold and log canonical threshold coincide.
\label{cond2}
\end{enumerate}
The example \cite[Example 4.5]{MustataTakagiWatanabeFThresholdsAndBernsteinSato} and our Theorem A shows that these are distinct conditions, and it is natural to ask whether there are other conditions that either imply or are implied by \eqref{cond1} and/or \eqref{cond2}.


Toward this end, we shift our focus toward the \emph{$F$-signature function},
which asymptotically counts certain Frobenius splittings associated
to a pair $(R,f)$.   It is important to recall that if $R$ is an $F$-finite local
ring, then this function is continuous and convex, so that
one-sided derivatives exist at all points
\cite[Theorems 3.2, 3.5]{BlickleSchwedeTuckerTestIdeals2}.  Furthermore, these
derivatives encode other important numerical invariants; e.g., if $R$ is a domain,
then the negative of the right derivative at zero is the Hilbert-Kunz
multiplicity of $R/f$, while the negative of the left derivative at one is the
(traditional) $F$-signature of $R/f$ \cite[Theorem 4.4]{BlickleSchwedeTuckerTestIdeals2}.
Motivated by this, and the fact the $F$-signature function is supported on the interval
$[0, \fpt(f)]$, it is therefore natural to consider the left derivative of the
$F$-signature function at the $F$-pure threshold.  We show that
either of the conditions \eqref{cond1} or \eqref{cond2} imply similar behavior of
the $F$-signature function at the $F$-pure threshold.

\begin{samepage}
\begin{theoremB*}[\autoref{thm.LeftDerivativeOfFSig},\autoref{thm.FsigDerivativeNonzeroIfFPT=LCT}]
 Suppose that $f$ is a square-free element of an $F$-finite regular local ring $R$ of characteristic $p>0$.
Suppose further that the $F$-pure threshold of $f$ is less than one, and one of the following two conditions holds:
\begin{itemize}
\item[(1)]  $p$ does not divide the denominator of the $F$-pure threshold, or
\item[(2)]  $\fpt(f) = \lct(f)$ and there exists a divisor $E$ on some birational model such that the discrepancy of $(R, f^{\lct(f)})$ is $-1$ along $E$.\footnote{Such an $E$ always exists assuming the existence of resolution of singularities.  }
\end{itemize}
Then the left derivative of the $F$-signature function associated to $(R,f)$ at the $F$-pure threshold of $f$ equals zero.
\end{theoremB*}
\end{samepage}

\noindent Note that the portion of this result showing that the derivative vanishes under condition \eqref{cond1} originally appeared in an unpublished preprint of the first author \cite{CantonLeftDerivativeOfFSignature}.

Finally, an appendix written by a separate set of authors is included.
The results therein demonstrate another way that conditions \eqref{cond1} and \eqref{cond2} differ.
In particular, the appendix provides an example in which 
condition (b) is satisfied, and hence the left derivative of the $F$-signature function is zero, but the test ideal is not radical.  Note that the test ideal is always radical under condition (a) \cite{FedderWatanabe,VassilevTestIdeals,SchwedeSharpTestElements}.

\vskip 6pt

\noindent\emph{Acknowledgements:}  The authors of this paper would like to thank Paolo Cascini, Mircea \mustata, Karen Smith, and Kevin Tucker for valuable conversations.  We would like to thank Bernd Schober and Susan M\"uller for pointing out typos in a previous draft of this paper.   We would also like to especially thank Shunsuke Takagi for pointing out \cite[Example 4.5]{MustataTakagiWatanabeFThresholdsAndBernsteinSato}.

\section{A family of examples}

The $F$-pure threshold of an element $f$ of an $F$-finite regular ring $R$ may be described as the supremum over all positive real parameters $\lambda$ such that $(R, f^{\lambda})$ is sharply $F$-pure, or equivalently, as the supremum over all $\lambda > 0$ such that $\tau(R, f^{\lambda}) = R$.  This invariant is always a positive rational number in the unit interval \cite{BlickleMustataSmithDiscretenessAndRationalityOfFThresholds}.

For the convenience of the reader, we review these notions below in a particularly interesting (and simple) setting:  Suppose that $(R, \m)$ is an $F$-finite regular local ring, and that $\lambda$ is a positive rational number whose denominator (in lowest terms) is not divisible by $p$.  In other words, $\lambda = a/(q-1)$, where $q$ is a power of $p$ and $a$ is some positive integer.  In this context, $(R, f^{\lambda})$ is sharply $F$-pure whenever $f^a \notin \m^{[q]}$  \cite[Corollary 3.4]{SchwedeSharpTestElements}, and $\tau(R, f^{\lambda})$ is the minimal ideal $\idealb$ of $R$ (with respect to inclusion) with $f^a \in (\idealb^{[q]} : \idealb)$ \cite[Theorem 6.3]{SchwedeCentersOfFPurity}.  It is important to note that both of these notions depend only on the parameter $\lambda$, and not on the particular representation $\lambda = a/(q-1)$.   Finally, we recall that $\tau(R, f^{\fpt(f)})$ is a proper ideal containing $f$, and that $(R, f^{\fpt(f)})$ is sharply $F$-pure if and only of the denominator of $\fpt(f)$ is not divisible by $p$ (see, e.g., \cite[Remark 5.5]{SchwedeSharpTestElements} or \cite[Theorem 4.1]{HernandezFPurityOfHypersurfaces}).

\subsection{Some characterizations}

Below,  we characterize when the $F$-pure threshold of a hypersurface in an $F$-finite regular local ring has a certain special form.

\begin{setup}
\label{characterization: S}  Suppose that $f$ is an element of an $F$-finite regular local ring $(R, \m)$, and that $\lambda = a/(q-1)$, where $q$ is a power of $p$ and $a$ is some positive integer.
\end{setup}

The following definition is \emph{not} required to understand most of the results in this paper (especially when working over a regular ambient ring).  However, some find this framework convenient to work with.  We will point out at various points how this language implies things we also prove directly.

\begin{definition}[Uniformly $F$-compatible ideals]
An ideal $I$ of $R$ (which is not necessarily regular) is {\em uniformly $(f^t, F)$-compatible} if for every $R$-linear map
\[ \phi: F^e_*R \to R, \]
we have that $\phi(f^{\lceil t(p^e-1)\rceil}I) \subseteq I$.
\end{definition}

When $R$ is regular and local,
this is equivalent to the requirement that \mbox{$f^{\lceil t(p^e-1) \rceil} \in (I^{[p^e]}:I)$} for all $e \ge 0$
\cite[Proposition 3.11]{SchwedeCentersOfFPurity}.
For $\lambda = a/(q-1)$, to check that $I$ is uniformly $(f^\lambda, F)$-compatible, it suffices to
verify that $f^a \in (I^{[q]}:I)$. Therefore, the test ideal
$\tau(R, f^\lambda)$ is the unique smallest uniformly $(f^\lambda, F)$-compatible ideal
of $R$. If a pair $(R, f^t)$ is sharply $F$-pure, then every uniformly $(f^t, F)$-compatible ideal is
radical \cite[Corollary 3.3]{SchwedeCentersOfFPurity}.

\begin{proposition}
\label{prop.characterization}
In the context of \autoref{characterization: S}, we have that $\fpt(f) = \lambda = a/(q-1)$ if and only if there exists a proper ideal $\idealb$ of $R$ 
such that $f^a  \in \icolon{\idealb}{\idealb^{[q]}} \setminus \m^{[q]}$.
In this case, $f \in \idealb$.
\end{proposition}

\begin{proof}
First, suppose that $\fpt(f) = \lambda$.  The form of  $\lambda$ implies that $(R, f^{\lambda})$ is sharply $F$-pure, so that $f^a \notin \m^{[q]}$, and we may then set $\idealb = \tau(R, f^{\lambda})$.  Next, suppose that $f^a \notin \m^{[q]}$, and that $f^a \in \icolon{\idealb}{\idealb^{[q]}}$ for some proper ideal $\idealb$.  This first condition implies that $(R, f^{\lambda})$ is sharply $F$-pure, and so $\fpt(f)\geq \lambda$, while the second condition, and the minimality of the test ideal, shows that $\tau(R, f^{\lambda})$ is contained in $\idealb$, and is therefore a  proper ideal.  The characterization of the $F$-pure threshold via test ideals then  shows that $\fpt(f)\leq \lambda$.  It remains to show that $f \in \idealb$.  However, the assumption that $f^a \notin \m^{[q]}$ implies that $\lambda \leq 1$, and as test ideals decrease as the parameter increases, we have that
$\langle f \rangle  = \tau(R, f^{1}) \subseteq \tau(R, f^{\lambda}) \subseteq \idealb$.
\end{proof}

One can also prove the above result using the language of uniformly $F$-compatible ideals.  In the $(\Leftarrow)$ direction, the point is that the $\idealb$ is uniformly $(f^{\lambda}, F)$-compatible and hence contains the test ideal.


%


Next, we obtain a refined statement in the case $f$ has an isolated singularity.
\begin{lemma}
\label{lem.SpeciaFedderfptCriterion}  In the context of \autoref{characterization: S}, if $x_1, \ldots, x_n$ is a system of parameters for $R$ and $\sqrt{\tau(R, f^{\lambda})} = \m$ (e.g., if the hypersurface defined by $f$ has an isolated singularity at $\m$), then $\fpt(f) = \lambda$ if and only if $f^a \equiv u (x_1 \cdots x_n)^{q-1} \bmod \m^{[q]}$ for some unit $u \in R$.
\end{lemma}
\begin{proof}
If $f^a$ satisfies the desired congruence modulo $\m^{[q]}$, then one may take $\idealb = \m$
in \autoref{prop.characterization}.  Next, assume that $\fpt(f) = \lambda$. Since $(R, f^\lambda)$
is sharply $F$-pure, $\tau = \tau(R, f^\lambda) = \sqrt{\tau} = \m$ and $f^a \in (x_1\cdots x_n)^{q-1}R
\setminus \m^{[q]}$ so $f^a \equiv u(x_1 \cdots x_n)^{q-1} \bmod \m^{[q]}$.
\end{proof}

\subsection{Some computations}

\begin{setup}
\label{computation: s}
  Fix integers $d$ and $n$ satisfying $d \geq n \geq 4$ or $d > n = 3$, and  \[p \nmid d  ( n(d-2) - d). \]
We also fix a perfect field $k$ of characteristic $p>0$, and set \[ f = x_1^d + \cdots + x_n^d + (x_1 \cdots x_n)^{d-2} \in S =  k[x_1, \cdots, x_n].\]
Finally, we use $R$ to denote the localization of $S$ at $\m =  \langle x_1, \cdots, x_n \rangle \subseteq S$.
\end{setup}

\begin{remark}
\label{isolatedSing: R}
In the context of \autoref{computation: s}, the identities
\[ d(n(d-2)-d) x_i^d = d(d-2)  f + (n(d-2) -2d + 2) x_i \frac{\partial f}{\partial x_i} - (d-2) \sum_{j \neq i } x_j \frac{\partial f}{\partial x_j} \]
and our assumption on $p$ shows that $\m = \sqrt{(f, \frac{\partial f}{\partial x_1}, \cdots,  \frac{\partial f}{\partial x_n})}$.  Consequently, the hypersurface defined by $f$ has an isolated singularity at the origin, so that $\fpt(S, f) = \fpt(R, f)$.
\end{remark}

\begin{proposition}
\label{prop.PNotDividingDenominatorExample}
In the context of \autoref{computation: s}, if $p \equiv -1 \bmod d$, then $\fpt(S,f) = {n (p-d+1)+d \over d(p-1) }$.
\end{proposition}

\begin{proof}  Set $a = (p-d+1)/d \in \mathbb{N}$.  According to \autoref{isolatedSing: R} and \autoref{lem.SpeciaFedderfptCriterion}, to show that $\fpt(S,f) = \fpt(R, f) = \frac{n a+1}{p-1}$, we must show that
\begin{equation} \label{expansion: e}
f^{na+1} = \sum_{\overset{s_1, \ldots, s_n, t \geq 0}{s_1 + \cdots + s_n + t\,=\,n a+1}} \binom{na+1}{s_1, \ldots, s_n, t} x_1^{ds_1 + t(d-2)} \cdots x_n^{ds_n +t(d-2)}.
\end{equation}
is congruent to $u (x_1 \cdots x_n)^{p-1} \bmod \m^{[p]}$ for some nonzero $u \in k$.  Our approach will be to show that the only summand in \eqref{expansion: e} not contained in $\m^{[p]}$ corresponds to the index $(s_1, \cdots, s_n, t) = (a, \cdots, a, 1)$; for this index, the associated monomial is
\[
\begin{array}{rl}
 (x_1^d \cdots x_n^d)^a (x_1 \cdots x_n)^{d-2}
=  (x_1 \cdots x_n)^{ (p-d+1) + (d-2)}
=  (x_1 \cdots x_n)^{ p-1},
\end{array}
\]
with coefficient $\binom{na + 1}{a, \cdots, a, 1}$,
which is nonzero modulo $p$ since $na+1 < p$ by our assumptions.

Towards this end, we begin by noting that if a term in \eqref{expansion: e} is not contained in $\m^{[p]}$, then $t = 0$ or $t=1$.  Indeed, for such a term, we must have that $ds_i + t(d-2) \leq p-1$ for each $1 \leq i \leq n$, and summing these inequalities shows that
\[
d (s_1+\cdots+s_n) + nt(d-2) \leq n(p-1).
\]
After substituting the identity $s_1 + \cdots + s_n = n a +1-t$ and isolating all terms with $t$ appearing on the left-hand side, we find that
$t(n(d-2) - d) \leq n(d-2) - d$.
The assumed conditions on $d$ and $n$ imply that $n(d-2)- d  > 0$, so that $t \leq 1$.

It remains to show that a term in \eqref{expansion: e} is in $\m^{[p]}$ if $t = 0$, or if $t=1$ and the index satisfies $(s_1, \cdots, s_n) \neq (a, \cdots, a)$.  However, in either case, it is easy to see that some $s_i \geq a+1$, so that the power of $x_i$ is at least $d(a + 1) = p+1$.
\end{proof}

We now turn our attention to the log canonical threshold.
First we recall the definition of a log canonical pair in our setting.

Suppose that $Y = \Spec R$ where $R$ is a regular local ring or polynomial ring.  In this case a $\bQ$-divisor is simply $\Delta = {a \over m} \Div(f)$ for some rational number ${a \over m}$ and some $0 \neq f \in R$.  Hence $(Y, \Delta)$ carries exactly the same information as $(R, f^{a \over m})$.  Furthermore, we can pick our canonical divisor $K_Y = 0$ and observe that $m \Delta = a \Div(f)$ is Cartier (in other words, $\Delta$ is $\bQ$-Cartier).

Now consider a birational map from a normal $X$, $\pi : X \to Y$.
In this case, the canonical divisor $K_X = K_{\pi}$ is an exceptional divisor which measures the Jacobian of the birational map $\pi$, see \cite[Section 2.4]{BenitoFaberSmithMeasuring}.  Write
\[
\Delta_X := {1 \over m} \pi^*(m\Delta) - K_X = {a \over m} \Div_X f - K_X,
\]
a $\bQ$-divisor that is supported on the union
of the strict transform $\pi_*^{-1}\Delta$ and the divisorial
component $E = \sum E_i$ of the exceptional locus  of $\pi$.

\begin{definition}
We say that $(Y, \Delta)$ is
{\em log canonical } if the coefficients of $\Delta_X$ are at most one for every
birational morphism $\pi: X \to Y$, with $X$ normal.
\end{definition}

The general condition of log canonicity is often impossible to verify since we need to check every birational morphism.
However, if $(Y, \Delta)$ admits a log resolution, the condition simplifies
greatly. Recall that a proper birational morphism $\pi: X \to Y$ of
varieties is called a \emph{log resolution} of $(Y, \Delta)$ if $X$ is smooth,
the exceptional set $E$ of $\pi$ is a divisor, and
$\supp(\pi^{\ast} \Delta) \cup \supp(E)$ is in simple normal crossings.
Then $(Y, \Delta)$ is log canonical if and only if the coefficients of $\Delta_X$ are at most one
for a single log resolution $\pi: X \to Y$ of $(Y, \Delta)$.  See, for instance, \cite[Section 2.3]{KollarMori} for futher discussion on this topic.

Now, $\lambda \Delta$ is $\bQ$-Cartier for every rational number $\lambda \geq 0$. We can then consider the set of all rational $\lambda \geq 0$ for which $(Y, \lambda \Delta)$ is log canonical. The supremum over all such
$\lambda$ is the {\em log canonical threshold} of $(Y, \Delta)$, denoted
$\lct(Y, \Delta)$.


\begin{proposition}
\label{prop.LctComputation}  In the context of \autoref{computation: s}, blowing up the origin in $\mathbb{A}^n_k$ provides a log resolution of $(\mathbb{A}^n_k, \Div(f))$.
\end{proposition}
\begin{proof}  Let $\pi: X  \to \mathbb{A}^n_k$ be the blowup of $\mathbb{A}^n_k$ at the origin, $E$ the exceptional divisor of $\pi$, and $D$ the strict transform of $\Div(f)$.  Since $\pi^{\ast} \Div(f) = D + dE$, it suffices to show that $D$ is smooth and that $D$ and $E$ intersect transversally.

By symmetry, it suffices to establish these facts on the affine chart $U$ of $X$ on which $\pi$ is given by the map $S \to k[x_1, y_2, \cdots, y_n]$ sending $x_1 \mapsto x_1$ and $x_i \mapsto x_1 y_i$ for $2 \leq i \leq n$.  On this chart, $E$ is defined by $x_1$ and $D$ is defined by
\[ g = 1 + y_2^d + \cdots + y_n^d + x_1^{n(d-2) - d} (y_2 \cdots y_n)^{d-2}.\]
Given these equations,  it is apparent that $D$ and $E$ intersect transversally on $U$ since $p$ does not divide $d$.  Moreover, setting $N = n(d-2)-d$, the easily-verified identity
\[ d N = d N g + (N-d+2) x_1 \frac{\partial g}{\partial x_1} - N \sum_{i=2}^n y_i \frac{\partial g}{\partial y_i}\]
and our assumption that $p \nmid dN$ implies that $g$ is smooth on $U$.
\end{proof}

\begin{corollary}
\label{cor.LCTComputed}  In the context of \autoref{computation: s}, the log canonical threshold of $(\mathbb{A}^n_k, \Div(f))$ equals $n/d$.  In particular, if $p \equiv -1 \bmod d$, then $\lct(\mathbb{A}^n_k, \Div(f)) \neq \fpt(S, f)$, yet the denominator of the latter is not divisible by $p$.
\end{corollary}

\begin{proof}  We adopt the notation used in the proof of \autoref{prop.LctComputation}.  It is well-known that $K_{\pi} = (n-1)E$, so that  $K_{\pi} - \lambda \cdot \pi^{\ast} \Div(f) = (n-1-\lambda d) E  - \lambda D$.  Consequently,  $(\mathbb{A}^n_k, \lambda \Div(f))$ is log canonical if and only if $0 < \lambda \leq n/d$.
\end{proof}

\section{The left derivative of the $F$-signature function at the $F$-pure threshold}



In this section, we consider the \emph{$F$-signature function}\footnote{If
  $\varphi(R, f^t)$ denotes the function in \cite[Definition 2.4]{MonskyTeixeiraPFractals1}
  (with $I=\m$ and $h=f$) or the one in \cite[Definition 1.1]{MonskyTeixeiraPFractals2},
  then it is easy to see that the $F$-signature function satisfies the identity
  \[ s(R, f^t) = 1 - \varphi(R,f^t). \]
  Thus, in the settings considered by Monsky and Teixeira, many of the
  properties of $s(R, f^t)$ recalled in this section follow from the
  corresponding properties for $\varphi(R, f^t)$ established in
  \cite{MonskyTeixeiraPFractals1, MonskyTeixeiraPFractals2}.}
associated to an element $f$ of an $F$-finite regular local ring
$(R, \m)$. We begin by summarizing the needed theory, directing
the interested reader to \cite{BlickleSchwedeTuckerFSigPairs1} and
\cite{BlickleSchwedeTuckerTestIdeals2} for a complete development
with historical context.
Set $a_e(t) = \lambda_R\left(R/(\m^{[p^e]}: f^{\lceil t(p^e-1) \rceil})\right)$.
The $F$-signature is defined \cite[3.11]{BlickleSchwedeTuckerTestIdeals2} as
\[ s(R, f^t) = \lim_{e \to \infty} a_e(t)/p^{e\dim(R)}. \]
Assuming that $(R, f^t)$ is sharply $F$-pure, then we define the \emph{splitting prime} $P = P(R, f^t)$ to be the largest ideal such that
\[
f^{\lceil t(p^e-1)\rceil} \in (P^{[p^e]} : P)
\]
for all $e \geq 0$ (in other words $P$ is the largest proper uniformly $(f^t, F)$-compatible ideal).  It is a prime ideal, see \cite[Definition 3.2]{AberbachEnescuStructureOfFPure}, \cite{SchwedeCentersOfFPurity} and \cite[2.12]{BlickleSchwedeTuckerTestIdeals2} for further discussion.  By
\cite[4.2]{BlickleSchwedeTuckerTestIdeals2},
\[ 0 < \lim_{e \to \infty} \frac{a_e(t)}{p^{e\dim(R/P)}} \le 1. \]
The limit above is called the {\em $F$-splitting ratio} $r_F(R, f^t)$.
By definition, $s(R, f^t) \le r_F(R, f^t)$.
\begin{itemize}
\item{} If $t < \fpt(f)$, then
$P = 0$, so the $F$-signature and the $F$-splitting ratio agree.
\item{} When $t = \fpt(f)$ and $(R, f^t)$ is sharply $F$-pure\footnote{This implies that denominator of $t$ is not divisible by $p$ \cite{SchwedeSharpTestElements,HernandezFPurityOfHypersurfaces}.}, $\dim(R/P) < \dim(R)$ and
$s(R, f^{\fpt(f)}) = 0$.
\item{} Finally, if $(R, f^t)$ is not sharply $F$-pure (for instance if $t > \fpt(f)$), then
$f^{\lceil t(p^e-1) \rceil} \in \m^{[p^e]}$, so $a_e(t) = 0$.
\end{itemize}
Summarizing,
$s(R, f^t) > 0$ for all $t < \fpt(f)$ and $s(R, f^t) = 0$ for $t \ge \fpt(f)$.
In the case that $t = a/q$, the $F$-signature is computed as
\begin{equation}
\label{specialInputFSignature: e}
s(R, f^{t}) = \frac{ \lambda_R\left(R / (\frm^{[q]} : f^a )\right) }{ q^{\dim(R)}},
\end{equation}
which does not depend on the  particular representation $t=a/q$
\cite[Proposition 4.1]{BlickleSchwedeTuckerTestIdeals2}.

By \cite{BlickleSchwedeTuckerTestIdeals2}, all one-sided derivatives
of $s(R, f^t)$ exist.  In this section, we show that the left
derivative at $t=\fpt(f)$ equals zero whenever the $F$-pure threshold
is ``mild."  We note that  \autoref{lem.heightSplittingPrime} and
\autoref{thm.LeftDerivativeOfFSig} originally appear in the unpublished
manuscript \cite{CantonLeftDerivativeOfFSignature}.

\begin{lemma}\label{lem.heightSplittingPrime}
  Suppose that $f$ is a square-free element of an $F$-finite regular
  local ring $(R, \m)$.  If $q$ is a power of $p$ and $a$ is some positive
  integer less than $q-1$,  then the height of any ideal $\idealb$ of $R$
  containing $f$ with $f^a \in (\idealb^{[q]}: \idealb)$ is at least two.
\end{lemma}

\begin{proof}
By way of contradiction, assume that there exists a minimal prime $P$ of $\idealb$ with $\operatorname{ht}(P) = 1$.  As $P$ is prime, the containment $f \in P$ allows us to write $f = f_1 \cdots f_r$ as a product of distinct irreducible factors with $f_1 \in P$.  However, being a height one prime in a regular local ring, $P$ is principal, and must therefore be generated by the element $f_1$.  Moreover, the irreducibility of $f_1$ and the assumption that $a \leq q-2$ implies that
\[ f^a \cdot P =
f^a \cdot \langle f_1 \rangle = \langle f_1^{a+1} f_2^a \cdots f_r^a \rangle \not \subseteq  \langle f_1^q \rangle = P^{[q]}.
\]
Thus, $f^a \notin \icolon{P}{P^{[q]}}$, contradicting \cite[Proposition 4.10]{SchwedeCentersOfFPurity}.
\end{proof}

\begin{lemma}
\label{vanishingDerivativeRestatement: L}
Suppose $f$ is an element of a regular $F$-finite local ring $(R, \m)$.  If $q$ is a power of $p$, then the left derivative of $s(R, f^t)$ at $t = \fpt(f)$ equals zero if and only if
\[ \lim_{e \to \infty} \frac{ \lambda_R (R/ (\m^{[q^e]} : f^{\lceil q^e \fpt(f)  \rceil - 1}) )}{q^{e(\dim(R)-1)}} = 0.\]
\end{lemma}

\begin{proof}  Set $\alpha = \fpt(f)$.  As the sequence $\alpha_e = \frac{ \lceil  \alpha q^e \rceil - 1 }{q^e}$ converges to $\alpha$ from below, the fact that $s(R,f^{\alpha}) = 0$ and \eqref{specialInputFSignature: e} allow us to realize this left-derivative as
\begin{equation}
\label{leftDerivativeRestatement: e}
 \lim_{e \to \infty} \frac{s(R, f^{\alpha_e})}{\alpha_e - \alpha} = - \lim_{e \to \infty} \frac{ \lambda_R( R/ (\m^{[q^e]} : f^{\lceil \alpha q^e \rceil - 1}))}{q^{e(\dim R - 1)} \cdot \beta_e},
 \end{equation}  where $\beta_e = \alpha q^e - \lceil \alpha q^e \rceil +1$.   To complete the proof, it suffices to observe that $\beta_e$ is a bounded sequence that is bounded away from zero (indeed, if $d$ is a denominator for the rational number $\alpha$,  it is straightforward to verify that $1/d \leq \beta_e \leq 1$ for every $e \geq 1$).
\end{proof}

\begin{remark}
\label{leftDerivativeRestatement: R}
In the context of \autoref{vanishingDerivativeRestatement: L}, suppose that $\fpt(f) = a/(q-1)$ for $q$ a power of $p$ and $a$ a positive integer.  Setting $\delta_e = \frac{q^e-1}{q-1}$ and substituting the identity $\lceil q^e \fpt(f) \rceil  = a \delta_e + 1$ into \eqref{leftDerivativeRestatement: e} shows that the left derivative of $s(R,f^t)$ at $t=\fpt(f)$ equals \[ - \frac{1}{\fpt(f)} \cdot \lim \limits_{e \to \infty} \frac{\lambda_R\left(R/(\m^{[q^e]} : f^{a \delta_e })\right)}{{q^{e (\dim(R)-1)}}}.\]
\end{remark}

\begin{theorem}
\label{thm.LeftDerivativeOfFSig}
 Suppose that $f$ is a square-free element of an $F$-finite regular local ring $(R, \m)$.
If the $F$-pure threshold of $f$ is less than one, and $p$ does not divide its denominator, then the left derivative of $s(R, f^t)$ at $t=\fpt(f)$ is zero.
\end{theorem}

\begin{proof}  Write $\fpt(f) = \frac{a}{q-1}$, and set $\delta_e = \frac{q^e-1}{q-1}$.  By \autoref{prop.characterization}, there exists $f \in \idealb \subseteq \m$ such that $f^a \cdot \idealb \subseteq \idealb^{[q]}$.  Inducing on $e$ shows that $f^{a \delta_e} \cdot \idealb \subseteq \idealb^{[q^e]}$ for all $e \geq 1$, so that $\idealb + \m^{[q^e]} \subseteq (\m^{[q^e]} : f^{a \delta_e})$ for every $e \geq 1$.  Setting $A = R/\idealb$, this and \cite{MonskyHKFunction} show that  \[ \lambda_R  ( R / (\m^{[q^e]} : f^{a \delta_e}) ) \leq \lambda_R ( R / (\idealb + \m^{[q^e]}) ) = \lambda_A ( A / \m^{[q^e]} A) = e_{\operatorname{HK}}(A) \cdot q^{e \dim A} + \epsilon_q, \] where $e_{\operatorname{HK}}(A)$ is the Hilbert-Kunz multiplicity of $A$ and $\epsilon_q = O(q^{e (\dim A -1)})$.   To conclude the proof, simply note that  $\dim A \leq \dim R - 2$ by \autoref{lem.heightSplittingPrime}, so that the limit in  \autoref{leftDerivativeRestatement: R} equals zero.
\end{proof}

The hypothesis that $f$ be square-free in \autoref{thm.LeftDerivativeOfFSig} is necessary, as we see below.

\begin{example}
If $f = x^2 y \in R= \mathbb{F}_p[[x,y]]$ with $p \ne 2$, then
$\fpt(f) = \frac{1}{2}$. However, as $f$ is a monomial, it is easy
to compute that the left derivative of $s(R, f^t)$ at
$t=\frac{1}{2}$ is $-1$:  \ Indeed, the expression
$a\delta_e$ in \autoref{leftDerivativeRestatement: R} equals
$(p^e - 1)/2$, and
\begin{align*}
  \m^{[p^e]}:f^{a\delta_e} &= \m^{[p^e]}:(x^{p^e-1}y^{(p^e-1)/2}) 
                           = (x, y^{(p^e+1)/2}).
\end{align*}
Therefore,
\[ \lambda_R\left(R/(\m^{[p^e]}:f^{a\delta_e})\right) = (p^e-1)/2, \]
so the limit in \autoref{leftDerivativeRestatement: R} equals $-1$.
\end{example}

The following refinement of the argument presented above appears in \cite{CantonLeftDerivativeOfFSignature}.

\begin{remark}[Additional statements involving splitting primes]
\label{CantonRemark: R}
We adopt the context of \autoref{vanishingDerivativeRestatement: L}.  Given a positive integer $n$, the limit
\[ \ell_n(R,f) := \lim_{t \to \fpt(f)^-} \frac{s(R, f^{t})}{(t - \fpt(f))^n},\]
can be thought of as an ``approximate left $n^{\text{th}}$ derivative'' of the $F$-signature function at the $F$-pure threshold.  If $\fpt(f) = a/(q-1)$ and $\delta_e = \frac{q^e-1}{q-1}$, a straightforward generalization of \eqref{leftDerivativeRestatement: e} and \autoref{leftDerivativeRestatement: R} shows that if $\ell_n(R,f)$ exists, then
\begin{equation}
\label{quasiDerivative: e}
 \ell_n(R,f) = \left(- \frac{1}{\fpt(f)} \right)^n \cdot \lim_{e \to \infty} \frac{ \lambda_R (R/ (\m^{[q^e]} : f^{a \delta_e}))}{q^{e(\dim R - n)}}.
 \end{equation}

The vanishing of the limit in \eqref{quasiDerivative: e} is determined
by the height of the splitting prime $P = P(R, f^{\fpt(f)})$.
In light of \eqref{quasiDerivative: e},
\cite[Definition 4.5]{BlickleSchwedeTuckerFSigPairs1} may be restated as
 \begin{equation}
 \label{dimSplittingPrime: e}
\operatorname{ht} P = \min \left\{ n \in \mathbb{N} \ : \ \lim_{e \to \infty} \frac{\lambda_R\left( R/ (\frm^{[q^e]} : f^{ a \delta_e }  \right) }{q^{e(\dim R - n)}} \neq 0 \right\} = \min \{ n : \ell_n(R,f) \neq 0 \},
\end{equation}
where in the last equality, we assume that $\ell_n(R,f)$ exists for every $n \in \mathbb{N}$.

Observe that \eqref{quasiDerivative: e} and the first equality in
\eqref{dimSplittingPrime: e}  give another proof of
\autoref{thm.LeftDerivativeOfFSig}:  in this case,
\autoref{lem.heightSplittingPrime} shows that $\operatorname{ht} P \geq 2$,
and so $\ell_1(R,f)$ (which always exists) equals zero.  Finally, we observe
that by \eqref{quasiDerivative: e},
\[ \ell_{\operatorname{ht} P}(R,f) =  \left(- \frac{1}{\fpt(f)} \right)^{\operatorname{ht} P} \cdot \lim_{e \to \infty} \frac{ \lambda_R (R/ (\m^{[q^e]} : f^{a \delta_e}))}{q^{e(\dim R/P)}} = \frac{r_F(R, f^{\fpt(f)})}{\left( - \fpt(f) \right)^{\operatorname{ht} P}}. \]
\end{remark}

In the final result of this section, we show that the left derivative of the F-signature function at the $F$-pure threshold also vanishes whenever the $F$-pure threshold agrees with the log canonical threshold.  In preparation for \autoref{thm.FsigDerivativeNonzeroIfFPT=LCT}, we recall some standard notation and basic facts:  If $R$ is an arbitrary ring of characteristic $p>0$, then $F^e_{\ast} R$ will denote the $R$-module obtained from restriction of scalars via the $e^{\text{th}}$ iterate of the Frobenius map.  Given an element $x \in R$, we denote the corresponding element in $F^e_{\ast} R$ by $F^e_{\ast} x$; in this notation, $r F^e_{\ast} x = F^e_{\ast} (r^{p^e} x)$ for every $r,x \in R$.  If $R$ is a domain with fraction field $\mathbf{K}$, then any map $\phi \in \operatorname{Hom}_R(F^e_{\ast} R, R)$ extends to one in $\operatorname{Hom}_{\mathbf{K}}(F^e_{\ast} \mathbf{K}, \mathbf{K})$ via the rule
\[ \phi \left( F^e_{\ast} \left(\frac{x}{y} \right) \right):= \frac{ \phi( F^e_{\ast}(y^{p-1} x))}{y}. \]
Finally recall that if $(R, \m)$ is $F$-finite and regular, then for every $e \geq 1$ and $g \in R$,
\begin{equation}
\label{colonViaHom: e}
g \in \m^{[p^e]} \text{ if and only if } \phi( F^e_{\ast} g) \subseteq \m \text{ for every $\phi \in \operatorname{Hom}_R(F^e_{\ast} R, R)$.}
\end{equation}

\begin{theorem}
\label{thm.FsigDerivativeNonzeroIfFPT=LCT}
Suppose that $f$ is an element of an $F$-finite 
regular local ring $(R, \frm)$ of dimension at least two. Assume,
further, that there exists a prime {exceptional} divisor
$E$ of a proper birational morphism $\pi : Y \to X = \Spec(R)$
with $Y$ normal such that $\frm \in \pi(E)$, and that the $E$-coefficient of
$K_Y - \fpt(f) \cdot  \Div_Y(f)$ is  $-1$.
(Note these hypotheses hold if a log resolution of singularities
exists, $\fpt(f) = \lct(f)<1$, and 
$\langle f \rangle$ is radical.)
In this case, the left derivative of $s(R, f^t)$ at $t=\fpt(f)$ is zero.
\end{theorem}

\begin{proof}
Since the $F$-signature cannot decrease after localization \cite[Proposition 1.3]{AberbachLeuschke}, we can assume that $\frm = \pi(E)$ after localizing $R$ at the generic point of $\pi(E)$ (since $E$ is exceptional, we still have that $\dim(R) \geq 2$).  Let $v$ denote the divisorial valuation, with valuation ring $R_{v}$ and uniformizer $r$, on the fraction field $\mathbf{K}$ of $R$ corresponding to $E$, and for every positive number $\gamma$, let  $\m_{v \geq \gamma} = r^{\lceil \gamma \rceil} R_v$ consist of all fractions whose value is at least $\gamma$.  Note that as $\pi(E) = \m$, we have that $\m_{v \geq 1} \cap R \subseteq \m$.

The key technical point of this proof is the following claim.
\begin{claim}
 If $\alpha = \fpt(f) = \lct(f)$, then
\[ \phi \left( F^e_{\ast} (f^{\lceil \alpha p^e \rceil} \cdot \m_{v \geq 1}) \right) \subseteq \m_{v \geq 1} \text{ for every $e \geq 1$ and $\phi \in \operatorname{Hom}_R(F^e_{\ast} R, R)$.}
\]
\end{claim}
\begin{proof}[Proof of claim]
Consider such a $\phi$.  As in \cite[Section 4]{BlickleSchwedeSurveyPMinusE}, $\phi$ yields a $\bQ$-divisor $\Delta_{\phi} \geq 0$.  We consider the new $\bQ$-divisor $\Delta' = \Delta_{\phi} + {\lceil p^e \alpha \rceil \over p^e - 1}\Div(f)$.  Because the $E$ coefficient of
\[
K_Y - \alpha \cdot \Div_Y(f) = K_Y - \pi^*(\alpha \cdot \Div_X(f))
\]
equals -1, we see that $\beta := \coeff_E(K_Y - \pi^* \Delta')$ is $\leq -1$.  Define a new map on the fraction field $\psi(F^e_* \blank) := \phi(F^e_* f^{\lceil \alpha p^e \rceil} \blank)$.  Note that $\psi|_{F^e_* R}$ corresponds to the divisor $\Delta'$.  Hence, by \cite[Lemma 7.2.1]{BlickleSchwedeSurveyPMinusE}, $\psi|_{F^e_* R_v}$ corresponds to the divisor $-\beta \cdot \Div_{R_v}(r) \geq \Div_{R_v}(r)$.

Thus if $\Phi_v \in \Hom_{R_v}(F^e_* R_v, R_v)$ generates $\Hom_{R_v}(F^e_* R_v, R_v)$ as an $F^e_* R_v$-module, we can write $\psi_v(F^e_* \blank) = \Phi_v(F^e_* u r^{-\beta(p^e-1)} \blank)$ for some unit $u \in R_v$.  We then have that
\[
\begin{array}{rl}
& \phi \left( F^e_{\ast} (f^{\lceil \alpha p^e \rceil} \cdot \m_{v \geq 1}) \right) \\
=&  \psi\left( F^e_{\ast} \m_{v \geq 1} \right) \\
=& \Phi_v\left( F^e_{\ast} u r^{-\beta(p^e-1)} \m_{v \geq 1} \right)\\
 \subseteq & \Phi_v(F^e_* r^{p^e} R_v) \\
 \subseteq & r R_v\\
 = & \frm_{v \geq 1}.
\end{array}
\]
This proves the claim.
\end{proof}
As $\m_{v \geq 1} \cap R \subseteq \m$,  the above and \eqref{colonViaHom: e} then imply that
\[ (f^{\lceil \alpha p^e \rceil} \cdot \m_{v \geq 1} ) \cap R \subseteq \m^{[p^e]} \text {for every $e \geq 1$}.\]
Next, note that if $g \in R$ satisfies $v(g) \geq v(f) + 1$,  then $g/f \in \m_{v \geq 1}$, and thus
\[ g f^{\lceil \alpha p^e \rceil-1} = f^{\lceil \alpha p^e \rceil} \cdot g/f \in (f^{\lceil \alpha p^e \rceil} \cdot \m_{v \geq 1} ) \cap R \subseteq \m^{[p^e]} \text{ for every $e \geq 1$}. \]
This shows that $\m_{v \geq v(f) + 1} \cap R$ is contained in $( \m^{[p^e]} : f^{ \lceil \alpha  p^e \rceil-1})$, and therefore the length of $R/(\m^{[p^e]} : f^{\lceil \alpha p^e \rceil-1})$ is bounded above by the length of $R/( \m_{v \geq v(f) + 1} \cap R)$ for every $e \geq 1$.  The theorem then follows from \autoref{vanishingDerivativeRestatement: L} (here, it is important that $\dim R \geq 2$).
\end{proof}
%
%
%

We conclude by highlighting a few questions arising from our investigation.

\begin{question}
  If the left derivative of the $F$-signature function at the $F$-pure threshold vanishes, does this guarantee any ``nice'' behavior (e.g., from the point of view of any of the well-studied singularities defined via Frobenius)?
\end{question}

\begin{question}
  Do the higher (left) derivatives of the $F$-signature function exist at the $F$-pure threshold?  If so, how are they related to the approximations considered in \autoref{CantonRemark: R}.
\end{question}

\begin{question}
  Do all of the results of this section hold when the ambient ring is not regular?  Note \autoref{thm.FsigDerivativeNonzeroIfFPT=LCT} does.
\end{question}

\appendix\section{Another interesting example}
\begin{center}
  {by \textsc{Alessandro De Stefani, Jack Jeffries, Zhibek Kadyrsizova, Robert Walker, and George Whelan} }
\end{center}
\vskip 6pt
Consider the pair $(R, f^{\fpt(f)})$.  In the introduction, the authors discussed two conditions which seem to imply that this pair is ``well behaved.''
\begin{enumerate}
  \item{}  The characteristic does not divide the denominator of the $F$-pure threshold ($\fpt$).
  \item{}  The $F$-pure threshold ($\fpt$) and log canonical threshold ($\lct$) coincide.
\end{enumerate}
\cite[Example 4.5]{MustataTakagiWatanabeFThresholdsAndBernsteinSato} and the paper above show that these conditions are distinct.
It was also shown that both conditions imply certain behavior of the $F$-pure threshold.  It is thus natural to ask whether there are any other conditions that are implied by these.  If the characteristic does not divide the denominator of the $\fpt$, then the pair $(R, f^{\fpt(f)})$ is sharply $F$-pure and hence the corresponding test ideal $\tau(R, f^{\fpt(f)})$ is radical.  One might then hope that if the $F$-pure threshold and log canonical threshold coincide, then the test ideal is likewise radical.

The purpose of this appendix is to exhibit examples where $\fpt{} = \lct{}$ but the test ideal is not radical. In fact, we are able to produce a family of such examples, indexed by $n \in \NN$, where the the length of $R/\tau(R, f^{\fpt(f)})$ increases as $n$ increases. Our examples are inspired by the examples of Musta\c{t}\u{a} and Yoshida in \cite{MustataYoshidaTestIdealVsMultiplierIdeals}.

Let $n \gs 2$ be an integer, and let $N=2n+1$. Consider $R=\FF_2[x,y]$ and set \begin{equation}
  \label{eq.AppendixExample}
  f=x^2y^2+x^N+y^N \in R.
\end{equation}
By \cite{BlickleMustataSmithDiscretenessAndRationalityOfFThresholds}, since the characteristic is $2$, we have that $\tau(f^{1/2}) = \langle f \rangle^{[1/2]}$.  We can write $f=(xy)^2 \cdot 1 + (x^n)^2\cdot x + (y^n)^2 \cdot y$, therefore we obtain that $\tau = \langle f\rangle^{[1/2]} = \langle x^n,xy,y^n\rangle$.  In particular, note that the test ideal $\tau$ is not radical.  Furthermore, we see that the length $\ell_R(R/\tau) = 2n - 1 = N - 2$ and so it even has unbounded length.

\begin{proposition} \label{ex_lct=fpt} With $f$ as in \eqref{eq.AppendixExample}, $\fpt(f) = 1/2 = \lct(f)$.
\end{proposition}
\begin{proof} Given that $\tau(f^{1/2}) = \langle x^n,xy,y^n \rangle \ne R$, we see that $\fpt(f) \leq 1/2$.  To show the other inequality, we prove that $(R, f^{1/2})$ is $F$-pure\footnote{Following \cite{HaraWatanabeFRegFPure}, this just means that $f^{\lfloor 1/2(p^e-1) \rfloor} \notin \m^{[2^e]}$ for $e \gg 0$.} or equivalently that $(R, f^{1/2 - \varepsilon})$ is sharply $F$-pure for $\varepsilon > 0$.  Indeed,
  \[
    \displaystyle \fpt(f) = \sup \{r \in \RR_{\geq 0} \mid \tau(f^r) = R\} = \sup\{t \in \RR_{\geq  0} \mid (R,f^t) \mbox{ is } F\mbox{-pure}\}.
  \]
  For all integers $e \geq 1$, we have that $\lfloor \frac{2^e-1}{2} \rfloor = 2^{e-1}-1$. It is easy to see that in the expansion of $f^{\lfloor \frac{1}{2} (2^e-1)\rfloor } =  f^{2^{e-1}-1}$, the term $(x^2y^2)^{2^{e-1}-1} = (xy)^{2^e - 2}$ has the smallest possible degree of any term.  Hence $(xy)^{2^e - 2}$ does not get canceled and it appears in the expansion of $f^{\lfloor \frac{1}{2}(2^e-1)\rfloor }$ with non-zero coefficient.  As $(xy)^{2^e-2} \notin \m^{[2^e]}$, we conclude that $f^{\lfloor \frac{1}{2}(2^e-1) \rfloor} \notin \m^{[2^e]}$. Therefore the pair $(R,f^{1/2})$ is $F$-pure, as claimed.  This shows that $\fpt(f) \geq 1/2$.

  Now we turn our attention to $\lct(f)$, the log canonical threshold.  Since $\lct(f) \geq \fpt(f) = 1/2$, it suffices to show that $1/2 \geq \lct(f)$.  To this end, blow up the origin to obtain $\pi : Y \to \mathbb{A}^2 = \Spec R$.  Note that the relative canonical divisor is simply one copy of the exceptional divisor $K_{Y / \mathbb{A}^2} = E$.  In the chart $\Spec \FF_2[\frac{y}{x},x] = \Spec \FF_2[u,x]$, we have that the pullback of $f$ is
  \[
    \ds x^4(u^2+x^{N-4}+u^Nx^{N-4}).
  \]
  By symmetry, we see that $\pi^* \Div(f) = 4E + \widetilde{H}$ where $\widetilde{H}$ is the strict transform of $\Div(f)$.  Note that $\widetilde{H}$ is defined by $u^2+x^{N-4}+u^Nx^{N-4}$ in the chart we wrote down.
  In order for $(R, f^{t})$ to be log canonical, we must have $\coeff_E(K_{Y / \mathbb{A}^2} - t \pi^* \Div(f)) \geq -1$.  By our previous computation, this is the same as requiring that
  \[
    1 - 4t \geq -1
  \]
  or in other words that $t \leq 1/2$.  It follows that
  \[
    \lct(f) = \sup\{t \in \bR_{\geq 0}\;|\; (R, f^t) \text{ is log canonical.}\} \leq 1/2.
  \]
  As discussed above, this completes the proof.
\end{proof}

In conclusion:

\begin{corollary}
There exist examples $f \in R$ where $\fpt(f) = \lct(f)$, and hence the derivative of the $F$-signature is zero by \autoref{thm.FsigDerivativeNonzeroIfFPT=LCT} but where $\tau(R, f^{\fpt(f)})$ is not radical.
\end{corollary}
\vskip 10pt
\noindent
{\it Acknowledgements:}  The authors of this appendix worked this example out at the Mathematics Research Community (MRC) in Commutative Algebra.  The authors would like to thank the staff and organizers of this MRC for the support provided.  The authors would also like to thank Eric Canton and Karl Schwede for useful conversations.


\bibliographystyle{skalpha}
\bibliography{MainBib}
\end{document}